\documentclass{article}
\usepackage[utf8]{inputenc}

\title{\textbf{Generalized Identities of Certain Continued Fractions}}

\author{Shaoxiong (Steven) Yuan}
\date{June 2019}

\usepackage{graphicx}
\usepackage{float}
\usepackage{amsthm}
\usepackage{amsmath}
\usepackage{caption}
\usepackage{amssymb}

\usepackage{hyperref}
\hypersetup{
    colorlinks=true,
    linkcolor=blue,
    filecolor=magenta,      
    urlcolor=cyan,
}

\newtheorem{theorem}{Theorem}
\newtheorem{corollary}{Corollary}

\begin{document}

\maketitle

\section{Abstract}

\indent In this article, we will discover some new generalized identity regarding continued fractions. We will connect the results to Fibonacci numbers and Lucas numbers. For all the proofs, we will use induction.

\section{Introduction}

\indent We begin with a study of some basic facts about the continued fractions and Fibonacci sequence.

Consider this expression:
$$
a_0 + \dfrac{1}{a_1+ \dfrac{1}{a_2+\dfrac{1}{\ddots+\dfrac{1}{a_n}}}}
$$

Let's define this form of expression (with integers $a_0$, $a_1$, $a_2$, ... , $a_n$ and 1 at every numerator) to be a regular or simple continued fraction (abbr. continued fraction). For succinctness, we simplify this complicated expression into the following notation: [$a_0$, $a_1$, ... , $a_n$].

For example,
$$
[2,3,7] = 2 + \dfrac{1}{3 + \dfrac{1}{7}} = \dfrac{51}{22}
$$
and
$$
[1,5,6,8] = 1 + \dfrac{1}{5 + \dfrac{1}{6 + \dfrac{1}{8}}} = \dfrac{302}{253}.
$$

Continued fractions can be divided into two cases: finite continued fractions (i.e. consists of finite amount of terms) and infinite continued fractions (i.e. consists of infinite amount of terms: $[a_0, a_1, ...]$). There are some nice identities about infinite continued fractions. For example,
$$
\sqrt{19} = [4,2,1,3,1,2,8,2,1,3,1,2,8,...] \ \textrm{(a period of 6)}
$$
$$
e = [2,1,2,1,1,4,1,1, ...] \ \textrm{(a period of 3 and add 2 to one term of each cycle)}
$$

However, this article will mainly discuss the identities and generalizations of finite continued fractions.

Let's define $f_n$ to be the number of different ways to tile a board of length $n$ with
squares (of width $1$) and dominoes (of width $2$). Thus, $f_4 = 5$, because there are five ways
to tile a board of length $4$, as shown here (we can also use $1$'s and $2$'s to represent the  squares and
the dominoes):

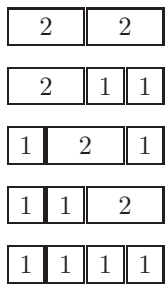
\setlength{\unitlength}{1.5pt}
\begin{picture}(100,90)
\multiput(11,11)(10,0){4}{\framebox(9,9){1}}
\put(11,26){\framebox(9,9){1}}
  \put(21,26){\framebox(9,9){1}}
  \put(31,26){\framebox(19,9){2}}
\put(11,41){\framebox(9,9){1}}
  \put(21,41){\framebox(19,9){2}}
  \put(41,41){\framebox(9,9){1}}
\put(11,56){\framebox(19,9){2}}
  \put(31,56){\framebox(9,9){1}}
  \put(41,56){\framebox(9,9){1}}
\put(11,71){\framebox(19,9){2}}
  \put(31,71){\framebox(19,9){2}}
\put(80,56){There are five ways to tile a strip}
\put(80,46){of length 4 with squares and dominoes.}
\put(80,30){Hence, $f_4 = 5$.}
\end{picture}

With a bit of effort, one is able to find these values for $f_n$:

\begin{tabular}{c|cccccccccccc}
\multicolumn{10}{c}{\rule[-5mm]{0mm}{12mm} \bf The first few values for $f_n$}\\
$n$ & 0 &1&2&3&4&5&6&7 &8 &9 &10 &11  \\
\hline
$f_n$ & 1 & 1 & 2 & 3 & 5 & 8 & 13 & 21 & 34 & 55 & 89 & 144 \\
\end{tabular}

\vskip0.3in

(One might wonder why we write $f_0$ as being equal to 1. This is because there is exactly one
way to tile a board of length 0, and that one way is {\em to use exactly 0 squares and 0 dominoes}.)
A quick glance at the above chart will lead to the discovery that each Fibonacci number is the sum of the two 
previous Fibonacci numbers. In other words, for $n \geq 2$,
$$
f_n = f_{n-1} + f_{n-2}.
$$

There's a delightful and well-known demonstration of this fact in the lovely book ``Proofs That Really Count" by Arthur T. Benjamin and Jennifer J. Quinn \cite{BQ}, 
which is full of many beautiful ways to prove mathematical formulas using visual proofs. 

In the book, the authors proved with visual tiling that if $[a_0, a_1, ... , a_n] = \dfrac{p_n}{q_n}$, then, for $n \geq 0$, $p_n$ is the ways to tile an $(n+1)$ tiling with height conditions $a_0, a_1, ... , a_n$, and $q_n$ is the ways to tile an $n$ tiling with height conditions $a_1, a_2, ... , a_n$ (All the tiles are covered by squares and dominoes, but only squares are stackable). Then immediately, by considering the condition of the last tile (square or domino), we have the following identity:
$$
p_n = a_np_{n-1} + p_{n-2}
$$
$$
q_n = a_nq_{n-1} + q_{n-2}
$$

Note that if we plug in $a_n = 1$ into both equation (i.e. $[1,1,1,...]$), we get
$$
p_n = p_{n-1} + p_{n-2}
$$
$$
q_n = q_{n-1} + q_{n-2}
$$

This is exactly the definition of Fibonacci numbers. 
$$
[1] = 1
$$
$$
[1,1] = \frac{2}{1}
$$
$$
[1,1,1] = \frac{3}{2}
$$
$$
[1,1,1,1] = \frac{5}{3}
$$
$$
[1,1,1,1,1] = \frac{8}{5}
$$

By writing out the first few terms, we can clearly see that this theorem holds true. From this example, we can see there are intimate connections between continued fractions and Fibonacci sequence. Inspired by this, we spend the whole summer in finding more delicate relationships in finite continued fraction and Fibonacci numbers. In the following discussion, we will cite some identities from the book for granted.

\section{Known Properties}

\indent There are several known properties of continued fractions in connection with Fibonacci sequence that lay the foundation for further research and discovery.

\textbf{Identity 1
(Identity 117 from ``Proofs that Really Count"):}

For $n \geq 0$,
$$
[4,4, ... ,4,3] = \frac{f_{3n+3}}{f_{3n}},
$$
where $a_n = 3$, and $a_i = 4$ for all $0 \leq i < n$.

We will prove this by induction.
\begin{proof}
\textbf{Step 1: Base Cases.}

$n = 0$: $[3] = \dfrac{3}{1} = \dfrac{f_3}{f_0} = \dfrac{f_{3\cdot0 + 3}}{f_{3\cdot0}}$

$n = 1$: $[4,3] = 4 + \dfrac{1}{3} = \dfrac{13}{3} = \dfrac{f_6}{f_3} = \dfrac{f_{3\cdot1 + 3}}{f_{3\cdot1}}$

$n = 2$: $[4,4,3] = 4 + \dfrac{1}{4 + \dfrac{1}{3}} = \dfrac{55}{13} = \dfrac{f_9}{f_6} = \dfrac{f_{3\cdot2 + 3}}{f_{3\cdot2}}$

True for $n = 0,1,2$

\textbf{Step 2: Inductive Step.}

Assuming that this identity is true for some ``$n$", we try to prove this formula is true for ``$n + 1$''.

Assumption: 
\[
[4,4, ... ,4,3] \ =\  4 + \dfrac{1}{4 + \dfrac{1}{\ddots +  \dfrac{1}{4+\dfrac{1}{3}}}} \ = \ \dfrac{f_{3n+3}}{f_{3n}},
\]
with $n$ 4's and one 3.

Then, we take the reciprocal of both sides. The equation becomes:
$$
 \dfrac{1}{4+\dfrac{1}{4+\dfrac{1}{\ddots+\dfrac{1}{4+\dfrac{1}{3}}}}}  = \dfrac{f_{3n}}{f_{3n+3}}
$$

Next, we add 4 to both sides and get:
$$
[4,4, ... ,4,4,3] \ \  = \ \  4 + \dfrac{1}{4+\dfrac{1}{4+\dfrac{1}{\ddots+\dfrac{1}{4+\dfrac{1}{3}}}}} \ \ = \ \ \dfrac{4f_{3n+3} + f_{3n}}{f_{3n+3}},$$
with $(n+1)$ 4's and one 3.

From Identity 18 in the book, we know
$$
4f_n = f_{n+2} + f_n + f_{n-2}.
$$ 

If we substitute $n$ with $3n+3$, we get
$$
4f_{3n+3} + f_{3n} = f_{3n+5} + f_{3n+3} + f_{3n+1} + f_{3n} = f_{3n+6}
$$ 
because of the definition of Fibonacci numbers.

This concludes our proof. 
\end{proof}

\textbf{Identity 2 (Identity 118 from the book):}

For $n \geq 0$,
$$
[4,4, ... ,4,5] = \frac{f_{3n+4}}{f_{3n+1}},  
$$
where $a_n = 5$, and $a_i = 4$ for all $0 \leq i < n$.

\begin{proof}
Proof by induction.

The proof is essentially the same as Identity 1 except that we need to substitute $n$ with $3n+4$ in order to get $3(n+1)+4 = 3n+7$.
\end{proof}

Before we move on to the next identity, we need to define a sequence first. Let's define $L_n$ be the sequence starting with $L_0 = 2$, $L_1 = 1$, $L_2 = 3$, and $L_n = L_{n-1} + L_{n-2}$. This sequence is actually called Lucas numbers. In the book ``Proofs that Really Count" \cite{BQ}, there exists a nice visual proof of Lucas numbers. Different from Fibonacci numbers, $L_n$ is the number of different ways to tile a bracelet of ``length" $n$ with
``squares" (of ``width" $1$) and ``dominoes" (of ``width" $2$). The first few values of $L_n$ are shown in the table below.

\begin{tabular}{c|cccccccccccc}
\multicolumn{10}{c}{\rule[-5mm]{0mm}{12mm} \bf The first few values for $L_n$}\\
$n$ & 0 &1&2&3&4&5&6&7 &8 &9 &10 &11  \\
\hline
$L_n$ & 2 & 1 & 3 & 4 & 7 & 11 & 18 & 29 & 47 & 76 & 123 & 199 \\
\end{tabular}

(Again, the reason why we write $L_0$ as 2 is because there are 2 ways to tile a bracelet of length 0, and those are: {\em use 0 squares and 0 dominoes in phase and out of phase}.)

\textbf{Identity 3:}

For $n \geq 0$,
$$
[4,4, ... ,4,7] = \frac{L_{3n+5}}{L_{3n+2}},
$$
where $a_n = 7$, and $a_i = 4$ for all $0 \leq i < n$.

\begin{proof}
Proof by induction.

\textbf{Step 1: Base Cases.}

$n = 0$: $[7] = \dfrac{7}{1} = \dfrac{L_5}{L_2} = \dfrac{L_{3\cdot0+5}}{L_{3\cdot0+2}}$

$n = 1$: $[4,7] = \dfrac{29}{7} = \dfrac{L_8}{L_5} = \dfrac{L_{3\cdot1+5}}{L_{3\cdot1+2}}$

$n = 2$: $[4,4,7] = \dfrac{123}{29} = \dfrac{L_{11}}{L_8} = \dfrac{L_{3\cdot2+5}}{L_{3\cdot2+2}}$

After verification, all base cases are true.

\textbf{Step 2: Inductive Step.}

Assuming that this theorem is true for some ``$n$", we want to extend it to the case ``$n+1$".

As what we've done in the proof of Identity 1, we take the reciprocal of the equation and add 4 to both sides. Finally, the equation becomes:
$$
[4,4, ... ,4,4,7] = \frac{4L_{3n+5} + L_{3n+2}}{L_{3n+5}}, \ \textrm{with $(n+1)$ 4's and one 7.}
$$

Suffice to show that $4L_{3n+5} + L_{3n+2} = L_{3(n+1) + 5} = L_{3n+8}$.

From Identity 32 in the book, 
$$
L_{3n+5} = f_{3n+5} + f_{3n+3}
$$
and
$$
L_{3n+2} = f_{3n+2} + f_{3n}.
$$

From Identity 18 again, 
$$
4f_{3n+5} = f_{3n+7} + f_{3n+5} + f_{3n+3}
$$
and
$$
4f_{3n+3} = f_{3n+5} + f_{3n+3} + f_{3n+1}.
$$

After converting them into Fibonacci numbers, the equation becomes
$$
f_{3n+7} + f_{3n+5} + f_{3n+3} + f_{3n+5} + f_{3n+3} + f_{3n+1} + f_{3n+2} + f_{3n}.
$$

Combining $f_{3n+7} + f_{3n+5} + f_{3n+3} + f_{3n+1} + f_{3n}$ and $f_{3n+2} + f_{3n+3} + f_{3n+5}$, we finally get $f_{3n+8} + f_{3n+6}$, which is equal to $L_{3n+8}$.

Thus, this completes the proof.
\end{proof}

\section{New Theorems}

\indent First, let's state the 3 known properties:
$$
\textrm{For $n \geq 0$}, \ [4,4, ... ,4,3] = \frac{f_{3n+3}}{f_{3n}}, \ \textrm{where $a_n = 3$, and $a_i = 4$ for all $0 \leq i < n$.}
$$
$$
\textrm{For $n \geq 0$}, \ [4,4, ... ,4,5] = \frac{f_{3n+4}}{f_{3n+1}}, \ \textrm{where $a_n = 5$, and $a_i = 4$ for all $0 \leq i < n$.}
$$
$$
\textrm{For $n \geq 0$}, \ [4,4, ... ,4,7] = \frac{L_{3n+5}}{L_{3n+2}}, \ \textrm{where $a_n = 7$, and $a_i = 4$ for all $0 \leq i < n$.}
$$

From the identities listed and proved above, we can observe a clear pattern. When the last number $j\in \mathbb{Z}$ equals 3 and 5, the resulting fraction is in the form $\dfrac{\textrm{Fibonacci}}{\textrm{Fibonacci}}$. When $j = 7$, the result will be $\dfrac{\textrm{Lucas}}{\textrm{Lucas}}$. Additionally, the index number of the numerator is always 3 more than that of the denominator ($3n+3$ and $3n$, $3n+4$ and $3n+1$, $3n+5$ and $3n+1$).

Then, we start to wonder what the result is when $j = 9$, since 9 is neither in Fibonacci numbers nor Lucas numbers.

Considering that Fibonacci sequence and Lucas sequence start with 1,1 and 2,1, respectively, we created a new sequence starting with 3,1. In fact, this sequence is the sum of the $(n-1)$th term of Fibonacci numbers and the $n$th term of Lucas numbers (i.e. $f_{n-1} + L_n$ for all $n\geq 0$) and is listed as 
\href{https://oeis.org/A104449}{A104449} 
in the Online Encyclopedia of Integer Sequences (OEIS).

\begin{tabular}{c|cccccccccccc}
\multicolumn{10}{c}{\rule[-5mm]{0mm}{12mm} \bf The first few values for $f_{n-1} + L_n$}\\
$n$ & 0 &1&2&3&4&5&6&7 &8 &9 &10 &11  \\
\hline
$f_{n-1} + L_n$ & 3 & 1 & 4 & 5 & 9 & 14 & 23 & 37 & 60 & 97 & 157 & 254 \\
\end{tabular}

Let's try the first few values to see if this proposition is true.
$$
n = 0: [9] = \frac{9}{1}
$$
$$
n = 1: [4,9] = \frac{37}{9}
$$
$$
n = 2: [4,4,9] = \frac{157}{37}
$$

Indeed, those results testified that our proposition is true. Curious readers may try out a few more values or let $j = 11,13...$ with new sequences starting with 4,1 and 5,1, correspondingly.

Define $G_n$ to be a linear recurrence sequence with $G_0^{(k)} = k$, $G_1^{(k)} = 1$, and $G_n^{(k)} = \textrm{the sum of the previous two terms}$ (i.e. $G_n^{(k)} = G_{n-1}^{(k)} + G_{n-2}^{(k)}$).

\begin{tabular}{c|cccccccccccc}
\multicolumn{10}{c}{\rule[-5mm]{0mm}{12mm} \bf The first few values for $G_n^{k}$}\\
$n$ & 0 &1&2&3&4&5&6&7 &8 &9 &10 &11  \\
\hline
$G_n^3$ & 3 & 1 & 4 & 5 & 9 & 14 & 23 & 37 & 60 & 97 & 157 & 254 \\
\hline
$G_n^4$ & 4 & 1 & 5 & 6 & 11 & 17 & 28 & 45 & 73 & 118 & 191 & 309 \\
\hline
$G_n^5$ & 5 & 1 & 6 & 7 & 13 & 20 & 33 & 53 & 86 & 139 & 225 & 364 \\
\hline
$G_n^6$ & 6 & 1 & 7 & 8 & 15 & 23 & 38 & 61 & 99 & 160 & 259 & 419 \\
\end{tabular}

With a bit of generalization and refinement, here is our new theorem:
\begin{theorem}\label{a}
For all $k\geq 0$ and $k\in \mathbb{Z}$, $G_0 = k$, $G_1 = 1$, and $G_n = G_{n-1} + G_{n-2}$. For all $m\geq 0 \in \mathbb{Z}$, $a_m = 2k+3$ and $a_i = 4$ for all $0\leq i < m$,
$$
[4,4, ... ,4,2k+3] = \frac{G_{3m+4}}{G_{3m+1}}.
$$
\end{theorem}

When $k = 0,1$, the $G_n$ sequence is Fibonacci numbers in disguise ($G_n^0 = F_n$ and $G_n^1 = F_{n+1}$ if $F_0 = 0$, $F_1 = 1$ and $F_n$ is the sum of two previous terms). In the same way, when $k = 2$, $G_n$ is Lucas numbers in disguise ($G_n^2 = L_n$ if $L_0 = 2$, $L_1 = 1$, and $L_n$ equals the sum of two previous terms). When $k = 3$, the continued fraction becomes $[4,4, ... ,4,9]$. After plugging in $m = 1,2,3$, the result is indeed $\dfrac{G_7}{G_4}, \dfrac{G_{10}}{G_7}, \dfrac{G_{13}}{G_{10}}$.

\begin{corollary}
Regardless of what value of $k$ is, the "Gibonomials" of the numerator and denominator, respectively, will have the same index for all $m$. 
\end{corollary}

For example, when $m = 2$, regardless of the value of $k$ (3, 4, 5, 6 from the chart above) is, $\dfrac{G_{3m+4}}{G_{3m+1}}$ will always equal $\dfrac{G_{10}}{G_7}$. When $m = 3$, the result will always be $\dfrac{G_{13}}{G_{10}}$.

Notice the following property. Any $G_n^{(k)}$ is constructed in a way that $G_n^{(k)} = f_n + k\cdot f_{n-1}$. Next, let's prove that this property satisfies the definition of the linear recurrence sequence: $G_n^{(k)} = G_{n-1}^{(k)} + G_{n-2}^{(k)}$.
$$
G_n^{(k)} = f_n + k\cdot f_{n-1}
$$
$$
G_{n-1}^{(k)} = f_{n-1} + k\cdot f_{n-2}
$$
$$
G_{n-2}^{(k)} = f_{n-2} + k\cdot f_{n-3}
$$

After inspection, we find out $G_n^{(k)}$ is indeed equal to $G_{n-1}^{(k)} + G_{n-2}^{(k)}$.

Based on this property, we can now present an even more concise and generalized theorem by avoiding using the new sequence $G_n$.

\begin{theorem}\label{b}
$F_n$ is the Fibonacci sequence defined with $F_0 = 0$, $F_1 = 1$, and of course $F_n = \textrm{sum of two previous terms}$. For all $k \geq 0 \in \mathbb{Z}$, $a_m = 2k+3$, $a_i = 4$ for all $0 \leq i < m$,
$$
[4,4, ... , 4,2k+3] = \frac{F_{3m+4} + k\cdot F_{3m+3}}{F_{3m+1} + k\cdot F_{3m}}.
$$
\end{theorem}
(Index numbers in the above equation can be flexible according to how you define the Fibonacci sequence.)

In order to prove this by induction,  
let's begin with analyzing the base cases of variable $m$.

\begin{proof}
\textbf{Base case: $m=0$}

So, set $m=0$ and let's
show that $[2k+3] = \dfrac{F_{3\cdot 0 + 4} + k F_{3\cdot 0 + 3}}{F_{3\cdot 0 + 1} + k F_{3\cdot 0 }} = \dfrac{F_4 + k\cdot F_3}{F_1 + k\cdot F_0}$ and since
$F_0 = 0, F_1 = 1, F_3 = 2, F_4 =3$,
we get the desired formula:
$$
[2k+3] = \frac{3 + 2k}{1 + 0k} = 2k+3
$$

\textbf{Now, induction on variable $m$}

Goal: We need to show that $[4,4, ... , 4,4,2k+3] = \dfrac{F_{3m+7} + k\cdot F_{3m+6}}{F_{3m+4} + k\cdot F_{3m+1}}$ for all constant $k,m \geq 0 \in \mathbb{Z}$, $a_{m + 1} = 2k + 3$, $a_i = 4$ for all $0 \leq i \leq m$, assuming that the theorem above is reasonable and verifiable at $m$.

In other words, we assume that 
\begin{equation}\label{1}
[4,4, ... , 4,2k+3] = \frac{F_{3m+4} + k\cdot F_{3m+3}}{F_{3m+1} + k\cdot F_{3m}}, \ \textrm{with $m$ 4's}
\end{equation}
and we want to prove that 
\begin{equation}\label{2}
[4,4, 4, ... , 4,2k+3] = \frac{F_{3m+7} + k\cdot F_{3m+6}}{F_{3m+4} + k\cdot F_{3m+3}}, \ \textrm{with $(m+1)$ 4's}
\end{equation}

We take the reciprocal of both sides
of equation (\ref{1}) and add 4 and get:
$$
\textrm{LHS} = \frac{F_{3m+1} + 4F_{3m+4} + k\cdot F_{3m} + 4k\cdot F_{3m+3}}{F_{3m+4} + k\cdot F_{3m+3}}
$$

Comparing the numerator and denominator, we find out that the denominator is the same as the goal. This suffices to prove that $F_{3m+1} + 4F_{3m+4} + k\cdot F_{3m} + 4k\cdot F_{3m+3} = F_{3m+7} + k\cdot F_{3m+6}$.

Using Identity 18 from ``Proofs that Really Count" again, we know $F_{3m+1} + 4F_{3m+5}$ is equal to $F_{3m+7}$ and $F_{3m} + 4F_{3m+3}$ is equal to $F_{3m+6}$. By simply moving the indices left 1 unit and multiplying by $k$, we can get the answer. Thus, this concludes our proof and we
have proved equation (\ref{2}). 
\end{proof}

With curiosity, we extended the range from only non-negative integers to all integers. The chart below shows several values when $k$ is negative.

\begin{tabular}{c|cccccccccccc}
\multicolumn{10}{c}{\rule[-5mm]{0mm}{12mm} \bf The first few values for $G_n^{k}$ ($k<0$)}\\
$n$ & 0 &1&2&3&4&5&6&7 &8 &9 &10 &11  \\
\hline
$G_n^{(-4)}$ & -4 & 1 & -3 & -2 & -5 & -7 & -12 & -19 & -31 & -50 & -81 & -131 \\
\hline
$G_n^{(-3)}$ & -3 & 1 & -2 & -1 & -3 & -4 & -7 & -11 & -18 & -29 & -47 & -76 \\
\hline
$G_n^{(-2)}$ & -2 & 1 & -1 & 0 & -1 & -1 & -2 & -3 & -5 & -8 & -13 & -21 \\
\hline
$G_n^{(-1)}$ & -1 & 1 & 0 & 1 & 1 & 2 & 3 & 5 & 8 & 13 & 21 & 34 \\
\end{tabular}

One thing worth mentioning is that $G_n^{(-1)}$ and $G_n^{(-2)}$ are negative Fibonacci numbers, and $G_n^{(-3)}$ are negative Lucas numbers. The same ``Gibonomials" are constructed according to the definition that $G_n^{(k)} = f_n + k\cdot f_{n-1}$ ($k$ is negative in these cases).

\begin{corollary}
$F_n$ is the Fibonacci sequence defined with $F_0 = 0$, $F_1 = 1$, and $F_n = \textrm{sum of two previous terms}$ ($n \geq 2$). For all $k \in \mathbb{Z}$, $a_m = 2k+3$, $a_i = 4$ for all $0 \leq i < m$,
$$
[4,4, ... , 4,2k+3] = \frac{F_{3m+4} + k\cdot F_{3m+3}}{F_{3m+1} + k\cdot F_{3m}}.
$$
\end{corollary}

Let's test a few values.

When $k = -4$ and $m = 2$, $[4,4,-5] = \dfrac{81}{19} = \dfrac{F_{10} - 4F_9}{F_7 - 4F_6} = \dfrac{G_{10}^{(-4)}}{G_7^{(-4)}}$ 

When $k = -3$ and $m = 2$, $[4,4,-3] = \dfrac{47}{18} = \dfrac{F_{10} - 3F_9}{F_7 - 3F_6} = \dfrac{G_{10}^{(-3)}}{G_7^{(-3)}}$ 

When $k = -2$ and $m = 2$, $[4,4,-1] = \dfrac{13}{5} = \dfrac{F_{10} - 2F_9}{F_7 - 2F_6} = \dfrac{G_{10}^{(-2)}}{G_7^{(-2)}}$

Even after we let $k$ equal to negative integers, those properties and formulas still apply.

After we tested and verified the properties above for the magical ``4'', we are motivated to try out other sets of values such as $[2,2, ... ,2,3]$, $[2,2, ... , 2,5]$ and so on, as well as $[3,3, ... , 3,3]$, $[3,3, ... ,3,5]$ and so on.

We will present our tests below.

For the number 2:
$$
[3] = \dfrac{3}{1}
$$
$$
[2,3] = 2+\dfrac{1}{3} = \dfrac{7}{3}
$$
$$
[2,2,3] = 2+\dfrac{1}{2+\dfrac{1}{3}} = \dfrac{17}{7}
$$

For the number 3:
$$
[3] = \dfrac{3}{1}
$$
$$
[3,3] = 3+\dfrac{1}{3} = \dfrac{10}{3}
$$
$$
[3,3,3] = 3+\dfrac{1}{3+\dfrac{1}{3}} = \dfrac{33}{10}
$$

After connecting them to existing sequences and creating some new ones according to the results given above, we cannot find any valid patterns that can be proved. They may apply in the first few values of numbers ``2" and ``3", but they can never be generalized like the number ``4".

Realizing that ``4" is one member of the Lucas numbers, we are suspicious that there might be another number that fits into a certain pattern.

Recall that in the introduction, we found an interesting identity about continued fraction: $[1,1, ... ,1] = \dfrac{F_{n+1}}{F_n}$. This leads us to wonder whether this pattern holds for more values like 3 and 5.

For the number 3:
$$
[3] = \dfrac{3}{1}
$$
$$
[1,3] = 1+\dfrac{1}{3} = \dfrac{4}{3}
$$
$$
[1,1,3] = 1+\dfrac{1}{1+\dfrac{1}{3}} = \dfrac{7}{4}
$$
(The numerators and denominators are in the Gibonacci sequence $G_n^2$)

For the number 5:
$$
[5] = \dfrac{5}{1}
$$
$$
[1,5] = 1+\dfrac{1}{5} = \dfrac{6}{5}
$$
$$
[1,1,5] = 1+\dfrac{1}{1+\dfrac{1}{5}} = \dfrac{11}{6}
$$
(The numerators and denominators are in the Gibonacci sequence $G_n^4$)

So far, readers may wonder whether $[1,1, ... ,1,2]$ and $[1,1, ... ,1,4]$ (even for the last number $a_n$, instead of odd) will have the same format of the final results. Let's test them patiently.

For the number 2:
$$
[2] = \dfrac{2}{1}
$$
$$
[1,2] = 1+\dfrac{1}{2} = \dfrac{3}{2}
$$
$$
[1,1,2] = 1+\dfrac{1}{1+\dfrac{1}{2}} = \dfrac{5}{3}
$$
(The numerators and denominators are indeed in the Gibonacci sequence $G_n^1$)

We discover that if the ending number is $k$, then the numerators and denominators are in the sequence $G_n^{(k-1)}$.

After generalization, we get:

For $k \geq 1 \in \mathbb{Z}$,
\begin{equation}\label{3}
[1,1, ... ,1,k] = \dfrac{G_{n+2}^{(k-1)}}{G_{n+1}^{(k-1)}},  
\end{equation}
where $a_m = k$, $a_i = 1$ for all $0 \leq i < m$.

Converting this into all-Fibonacci, we derive from equation (\ref{3}) our next theorem.

\begin{theorem}\label{c}
$F_n$ is the Fibonacci sequence defined with $F_0 = 0$, $F_1 = 1$, and $F_n = \textrm{sum of two previous terms}$. For all $k \geq 1 \in \mathbb{Z}$, $a_m = k$, $a_i = 1$ for all $0 \leq i < m$,
$$
[1,1, ... ,1,k] = \dfrac{F_{m+2}+(k-1)\cdot F_{m+1}}{F_{m+1}+(k-1)\cdot F_m}.
$$
\end{theorem}
(Again, index numbers can vary according to the definition of Fibonacci numbers.)

Next, let's prove it by induction.

\begin{proof}
\textbf{Base Case: $m = 0$}

Set $m = 0$ and show that $[k] = \dfrac{F_2+(k-1)\cdot F_1}{F_1+(k-1)\cdot F_0}$. Since, according to our definition, $F_0 = 0$, $F_1 = 1$, $F_2 = 1$, so
$$
[k] = \dfrac{1+(k-1)}{1} = k.
$$

\textbf{Now, induction on variable $m$}

We assume that
\begin{equation}\label{4}
[1,1, ... ,1,k] = \dfrac{F_{m+2}+(k-1)\cdot F_{m+1}}{F_{m+1}+(k-1)\cdot F_m}, \  \textrm{with $m$ 1's.}
\end{equation}

We want to demonstrate that
\begin{equation}\label{5}
[1,1, ... ,1,1,k] = \dfrac{F_{m+3}+(k-1)\cdot F_{m+2}}{F_{m+2}+(k-1)\cdot F_{m+1}}
\end{equation}
for all constant $m \geq 0 \in \mathbb{Z}$, $k \geq 1 \in \mathbb{Z}$, $a_{m + 1} = k$, $a_i = 1$ for all $0 \leq i \leq m$.

We take the reciprocal of both sides of equation (\ref{4}) and add 1 to get:
$$
[1,1, ... ,1,1,k] = \dfrac{F_{m+1}+F_{m+2}+(k-1)(F_m+F_{m+1})}{F_{m+2}+(k-1)\cdot F_{m+1}}
$$

It's easy to see that the numerators and denominators are the same as equation (\ref{5}).

This concludes our proof.
\end{proof}

We want to extend the realm valid sets of numbers into negative integers, just like Theorem 1.

Let's bring out the big chart again.

\begin{tabular}{c|cccccccccccc}
\multicolumn{10}{c}{\rule[-5mm]{0mm}{12mm} \bf The first few values for $G_n^{k}$ ($k<0 \in \mathbb{Z}$)}\\
$n$ & 0 &1&2&3&4&5&6&7 &8 &9 &10 &11  \\
\hline
$G_n^{(-4)}$ & -4 & 1 & -3 & -2 & -5 & -7 & -12 & -19 & -31 & -50 & -81 & -131 \\
\hline
$G_n^{(-3)}$ & -3 & 1 & -2 & -1 & -3 & -4 & -7 & -11 & -18 & -29 & -47 & -76 \\
\hline
$G_n^{(-2)}$ & -2 & 1 & -1 & 0 & -1 & -1 & -2 & -3 & -5 & -8 & -13 & -21 \\
\hline
$G_n^{(-1)}$ & -1 & 1 & 0 & 1 & 1 & 2 & 3 & 5 & 8 & 13 & 21 & 34 \\
\end{tabular}

\begin{corollary}
$F_n$ is the Fibonacci sequence defined with $F_0 = 0$, $F_1 = 1$, and $F_n = \textrm{sum of two previous terms}$. For all $k \in \mathbb{Z}$ and $m \geq 0$, $a_m = k$, $a_i = 1$ for all $0 \leq i < m$,
$$
[1,1, ... ,1,k] = \dfrac{F_{m+2}+(k-1)\cdot F_{m+1}}{F_{m+1}+(k-1)\cdot F_m}.
$$
\end{corollary}

In light of the fact that the numbers ``1", ``2", ``3", ``4" are all in the Lucas numbers sequence, we tested the number ``7", but it doesn't work. For the number ``11", however, when $a_n = 3$, some patterns appear.
$$
[3] = \dfrac{3}{1}
$$
$$
[11,3] = 11+\dfrac{1}{3} = \dfrac{34}{3}
$$
$$
[11,11,3] = 11+\dfrac{1}{11+\dfrac{1}{3}} = \dfrac{377}{34}
$$

The results turn out to be Fibonacci numbers.

\begin{theorem}\label{d}
$F_n$ is the Fibonacci sequence defined with $F_0 = 0$, $F_1 = 1$, and $F_n = \textrm{sum of two previous terms}$. For $m \geq 0 \in \mathbb{Z}$, $a_m = 3$, $a_i = 11$ for all $0 \leq i < m$,
\begin{equation}\label{6}
[11,11, ... ,11,3] = \dfrac{F_{5m+4}}{F_{5m-1}}.
\end{equation}
\end{theorem}
(Note: $F_{-1} = 1$ in the case $[3] = \dfrac{3}{1} = \dfrac{F_4}{F_{-1}}$)

\begin{proof}
Proof by induction:

\textbf{Base case of $m = 0$}

Set $m = 0$ and it's easy to show that $[3] = \dfrac{F_4}{F_{-1}}$.

\textbf{Now, induction on $m$}

We assume that equation (\ref{6}) is right and want to demonstrate that it is true for $(n+1)$, which is
\begin{equation}\label{7}
[11,11, ... ,11,11,3] = \dfrac{F_{5m+9}}{F_{5m+4}}.
\end{equation}

By taking the reciprocal of both sides of equation (\ref{6}) and adding 11, we only need to prove that $F_{5m+9} = F_{5m-1}+11F_{5m+4}$, which is equivalent to proving the equation
\begin{equation}\label{8}
F_{m+9} = F_{m-1}+11F_{m+4}.
\end{equation}

To prove this, let's decompose $11F_{m+4}$ first.

$11F_{m+4}$ can be divided into $2\cdot 4F_{m+4}+3F_{m+4}$. Using Identity 17 and 18 from the book, we get
$$
11F_{m+4} = 3F_{m+6}+2F_{m+4}+3F_{m+2}.
$$

Using Identity 17 on $3F_{m+6}$ and $3F_{m+2}$ again, we obtain $$
11F_{m+4} = F_{m+8}+4F_{m+4}+F_m.
$$

Using Identity 18 on $4F_{m+4}$, we get
\begin{equation}\label{9}
11F_{m+4} = F_m+F_{m+2}+F_{m+4}+F_{m+6}+F_{m+8}.
\end{equation}

Now, combine equation (\ref{9}) and $F_{m-1}$, we finally prove that
$$
11F_{m+4}+F_{m-1} = F_{m-1}+F_m+F_{m+2}+F_{m+4}+F_{m+6}+F_{m+8} = F_{m+9}.
$$

Since equation (\ref{8}) is proved, equation (\ref{7}) is proved, so our theorem is proved.
\end{proof}

Although the cases when $a_n = 3$ is fascinating, $a_n = 5$, $a_n = 7$, and $a_n = 9$ are not able to be generalized after testing. However, when $a_n = 11$, even more fascinating patterns start to show up.
$$
[11] = 11 = \dfrac{11}{1}
$$
$$
[11,11] = 11+\dfrac{1}{11} = \dfrac{122}{11}
$$
$$
[11,11,11] = 11+\dfrac{1}{11+\dfrac{1}{11}} = \dfrac{1353}{122}
$$

One might wonder that 122 and 1353 are not Lucas numbers or Fibonacci numbers, but once you take a closer look, you'll find out that $122 = 123-1$ and $1353 = 1364-11$. This time, 1, 123, 11, 1364 are all Lucas numbers. So it's easier for us to find a general formula.

For our theorem to be able to express, we defined a new Lucas sequence. Simply reversing the order of the first two numbers, we get the sequence $l_n$, defined with $l_0 = 1$, $l_1 = 2$, $l_2 = 3$, $l_3 = 4$, and $l_n = l_{n-1}+l_{n-2}$ for $n\geq 2$. (This sequence is actually the result of rounding $(\dfrac{1+\sqrt{5}}{2})^n$ to the nearest integer, which is also listed as \href{https://oeis.org/A169985}{A169985} 
in OEIS.) Readers might notice that $l_1+l_2 \neq l_3$. This is because we simply changed the position of the first 2 terms. In the original $L_n$ sequence, $L_0 = 2$, $L_1 = 1$, so $L_2 = 3$ and $L_3 = 4$. When $L_0$ and $L_1$ are exchanged, 3 and 4 remain unmoved, resulting in such a irregularity. Making such a tiny change is entirely for the purpose of expressing our theorem.

In the case $[11,11,11] = \dfrac{1353}{122}$, 1353 can be written into $123-1$, which is $l_{15}-l_5$, and 122 can be written into $123-1$, which is $l_{10}-l_0$. Consequently, we can state our new theorem.

\begin{theorem}\label{e}
Let $l_n$ be defined with $l_0 = 1$, $l_1 = 2$, $l_2 = 3$, $l_3 = 4$, and $l_n = l_{n-1}+l_{n-2}$ for $n\geq 2$ and let $a_i = 11$ for all $m \geq 0 \in \mathbb{Z}$ and $0 \leq i \leq m$, then
\begin{equation}\label{10}
[11,11, ... ,11,11] = \dfrac{l_{5m+5}-l_{5m-5}}{l_{5m}-l_{5m-10}}.
\end{equation}
\end{theorem}
(Special note: In the cases $[11,11]$ and $[11]$, where $m = 1$ and $m = 0$, respectively, the results contain $l_{-5}$ and $l_{-10}$, so we set them to be 0 in order for our theorem to be applicable to all non-negative $m$.)

Next, let's prove it by induction.

\begin{proof}
\textbf{Base Case: $m = 0$}

Set $m = 0$ and we want to show that $[11] = \dfrac{l_5-l_{-5}}{l_0-l_{-10}}$.

As we've explained earlier, this equation is true.

\textbf{Now, induction on variable $m$}

As usual, we take the reciprocal of both sides of equation(\ref{10}) and add 11 and get
$$
[11,11,11, ... ,11,11] = \dfrac{(l_{5m}+11l_{5m+5})-(l_{5m-10}+11l_{5m-5})}{l_{5m+5}-l_{5m-5}}.
$$

So, we only need to show that $$
l_{5m+10}-l_{5m} = (l_{5m}+11l_{5m+5})-(l_{5m-10}+11l_{5m-5}).
$$

According to our proof of Theorem 3, we can clearly see that the above equation is exactly true by alternating the index numbers.

Thus, we have proved Theorem 4.
\end{proof}

The next thing we wonder is if we can render it into Fibonacci numbers, just like Theorem 1,2, and 3.

This is not obvious, however, since 1353 and 122 are not in the original Fibonacci sequence. If we multiply 1353 and 122 by 5, we get 6765 and 610, which are $F_{20}$ and $F_{15}$.

Further simplifying Theorem 4, we obtain Theorem 5.

\begin{theorem}\label{f}
Let $F_0 = 0$, $F_1 = 1$ and $F_n = F_{n-1}+F_{n-2}$. Let $a_i = 11$ for all $m \geq 0 \in \mathbb{Z}$ and $0 \leq i \leq m$, then
$$
[11,11, ... ,11,11] = \dfrac{F_{5m+10}}{F_{5m+5}}.
$$
\end{theorem}
(Note: The result should be $\dfrac{\dfrac{F_{5m+10}}{5}}{\dfrac{F_{5m+5}}{5}}$. After we multiply the numerator and denominator by 5, those 5's cancel out.)

Let's test a few values before we go on to prove this.

When $m = 0$, $[11] = 11 = \dfrac{55}{5} =  \dfrac{F_{10}}{F_5}$.

When $m = 1$, $[11,11] = \dfrac{122}{11} = \dfrac{610}{55} = \dfrac{F_{15}}{F_{10}}$.

When $m = 2$, $[11,11,11] = \dfrac{1353}{122} = \dfrac{6765}{610} = \dfrac{F_{20}}{F_{15}}$.

Then, let's prove this by induction.

\begin{proof}
\textbf{Base cases: $m = 0,1,2$}

True.

\textbf{Now, induction on $m$}

After we execute the same procedure, we only need to show that
$$
F_{5m+5}+11F_{5m+10} = F_{5m+15}.
$$

We already proved this one.
\end{proof}

Besides, by referring to Theorem 3 and Theorem 4, we can get a new formula connecting Fibonacci numbers and adjusted Lucas numbers:
$$
l_m-l_{m-10} = \dfrac{1}{5}F_{m+5}.
$$

Actually, the numerators and denominators of the results of Theorem 4, for small to large, can form a sequence: 0, 1, 11, 122, 1353, ... This sequence is listed as
\href{https://oeis.org/A049666}{A049666}  
in OEIS. All terms have a generalized formula to represent them, which is $\dfrac{F_{5n}}{5}$, where $n$ is a non-negative integer.

Also, $[11,11, ... ,11,8]$, where $a_m = 8$ and $a_i = 11$ for $0 \leq i < m$, and $[11,11, ... ,11,13]$, where $a_m = 13$ and $a_i = 11$ for $0 \leq i < m$ can also be expressed in the form $\dfrac{\textrm{Fibonacci}}{\textrm{Fibonacci}}$. We'll let the readers to determine the index numbers and find the hidden patterns.

We have finished our discussion for the number ``11". We have reason to believe that for the numbers ``1" and ``4", the same pattern exists.

When $n = 1$, $F_{5n} = F_5 = 5$, and the denominator is exactly 5. 

We hypothesize that for the number ``1", the generalized formula will be $\dfrac{F_{1n}}{1}$. When $n = 1$, $F_{1n} = F_1 = 1$, which corresponds to the denominator. $\dfrac{F_{1n}}{1}$ is essentially $F_n$, which matches to Theorem 2.

For the number ``4", according to the format, the results form a sequence with a general formula $\dfrac{F_{3n}}{2}$ (listed as
\href{https://oeis.org/A001076}{A001076} 
in OEIS). When $n = 1$, $F_{3n} = F_3 = 2$, which is the same as the denominator. Here comes our next theorem.

\begin{theorem}\label{g}
Let $a_n = \dfrac{F_{3n}}{2}$ for $n \geq 0$. Let $a_i = 4$ for all $m \geq 0 \in \mathbb{Z}$ and $0 \leq i \leq m$, then
$$
[4,4, ... ,4,4] = \dfrac{a_{m+2}}{a_{m+1}}.
$$
\end{theorem}

\begin{proof}
Proof by induction:

\textbf{Base case: $m = 0$}

Set $m = 0$ and show that $[4] = \dfrac{a_2}{a_1}$. Since $a_2 = \dfrac{F_6}{2} = 4$ and $a_1 = \dfrac{F_3}{2} = 1$, $[4] = \dfrac{4}{1} = 4$.

\textbf{Induction on variable $m$}

Suffice to show $a_{m+1}+4a_{m+2} = a_{m+3}$.

Converting this into its original form and adjusting its index numbers as needed, we want to show that
$$
F_{3m}+4F_{3m+3} = F_{3m+6}
$$

Using Identity 18, we obtain that
$$
F_{3m}+F_{3m+1}+F_{3m+3}+F_{3m+5} = F_{3m+6},
$$
which is true.
\end{proof}

Now it's time for us to further explore the results when every number, or element, in the continued fraction is the same. Back to the original Lucas numbers where the sequence starts with 2, 1 and so on. From our discussion above, we know that ``1", ``4", and ``11" possess patterns, so we hypothesize that ``29" also has some kind of pattern, skipping ``18".

Before verifying, let's first guess the potential results. From $\dfrac{F_{1n}}{1}$, $\dfrac{F_{3n}}{2}$, and $\dfrac{F_{5n}}{5}$ for the number ``1", ``4", and ``11", respectively, we suppose the generalized formula $\dfrac{F_{7n}}{13}$ will fit for the number ``29".

For example,
$$
[29] = \dfrac{29}{1} = \dfrac{\dfrac{377}{13}}{\dfrac{13}{13}} = \dfrac{\dfrac{F_{14}}{13}}{\dfrac{F_7}{13}}
$$
and
$$
[29,29] = \dfrac{842}{29} = \dfrac{\dfrac{10946}{13}}{\dfrac{377}{13}} = \dfrac{\dfrac{F_{21}}{13}}{\dfrac{F_{14}}{13}}.
$$

\begin{theorem}\label{h}
Let $a_n = \dfrac{F_{7n}}{13}$ for $n \geq 0$. Let $a_i = 29$ for all $m \geq 0 \in \mathbb{Z}$ and $0 \leq i \leq m$, then
$$
[29,29, ... ,29,29] = \dfrac{a_{m+2}}{a_{m+1}}.
$$
\end{theorem}

\begin{proof}
Proof by induction:

\textbf{Base case: $m = 0,1$}

As tested above, all the base cases are true.

\textbf{Induction on variable $m$}

Suffice to show that $a_{m+1}+29a_{m+2} = a_{m+3}$. 

Converting the sequence $a_n$ into $F_n$, we only need to prove the equation
\begin{equation}\label{11}
F_m+29F_{m+7} = F_{m+14}.
\end{equation}

The first step is to decompose $29F_{m+7}$. From equation (\ref{8}), we get
$$
22F_{m+7} = 6F_{m+9}+4F_{m+7}+6F_{m+5}
$$

Then, we combine $4F_{m+7}$ and the remaining $7F_{m+7}$ and use equation (\ref{8}) again and get
$$
29F_{m+7} = 9F_{m+9}+2F_{m+7}+9F_{m+5}.
$$

The next step is to use Identity 17 from the book on $9F_{m+9}$ and $9F_{m+5}$, we obtain
$$
9F_{m+9} = 3F_{m+11}+3F_{m+7}
$$
and
$$
9F_{m+5} = 3F_{m+7}+3F_{m+3}.
$$

Next, we need to combine all the $F_{m+7}$ like terms and using Identity 18, we get
$$
8F_{m+7} = 2F_{m+9}+2F_{m+7}+2F_{m+5}.
$$

Then we use Identity 17 on $3F_{m+11}$ and $3F_{m+3}$ and combine all the like terms. For all the terms with coefficient 3, we use Identity 17. Similarly, for all the terms whose coefficients are 4, we use Identity 18. Finally, we get our desired formula:
\begin{equation}\label{12}
\begin{aligned}
&\ \ \ \ F_m+29F_{m+7} \\
&= F_m+F_{m+1}+F_{m+3}+F_{m+5}+F_{m+7}+F_{m+9}+F_{m+11}+F_{m+13} \\ 
&= F_{m+14}
\end{aligned}    
\end{equation}

Equation (\ref{12}) clearly shows that equation (\ref{11}) is true.

Thus, proof for Theorem 7 is finished.
\end{proof}

For Theorem 5, 6, 7, they all have something in common. First, the elements in those continued fractions are the same. Second, the results are render in the form $\dfrac{\textrm{Fibonacci}}{\textrm{Fibonacci}}$. However, those ``Fibonacci" numbers aren't directly extracted from the original sequence $F_n$. Instead, they have a hidden pattern: $\dfrac{F_{1n}}{1}$, $\dfrac{F_{3n}}{2}$, $\dfrac{F_{5n}}{5}$,
$\dfrac{F_{7n}}{13}$,
$\dfrac{F_{9n}}{34}$ and so on.

For the results to be valid, the numbers in the continued fraction must be inside Lucas numbers $L_n$, but not all of them satisfy. The truth is that there's one satisfying number every two Lucas numbers, as seen in the cases ``1", ``4", ``11, ``29". This implies that the next number that will work is 76.

Based on these findings, we have another corollary, which is the most generalized one.

\begin{corollary}
Let $L_n$ be the sequence that starts with $L_0 = 2$, $L_1 = 1$, and $L_n = L_{n-1}+L_{n-2} \ (n \geq 2)$. Let the sequence $a_n = \dfrac{F_{(2k+1)n}}{F_{2k+1}}$, where $k \geq 0 \in \mathbb{Z}$. Let $a_i = L_{2k+1}$ for all $m \geq 0 \in \mathbb{Z}$ and $0 \leq i \leq m$, then
$$
[L_{2k+1},L_{2k+1}, ... ,L_{2k+1},L_{2k+1}] = \dfrac{a_{m+2}}{a_{m+1}}.
$$
\end{corollary}

If $k = 4$, then $[76,76, ... ,76,76] = \dfrac{a_{m+2}}{a_{m+1}}$, where $a_n = \dfrac{F_{9n}}{34}$ (listed as 
\href{https://oeis.org/A049669}{A049669}
in OEIS).

If $k = 5$, then $[199,199, ... ,199,199] = \dfrac{a_{m+2}}{a_{m+1}}$, where $a_n = \dfrac{F_{11n}}{89}$ (listed as 
\href{https://oeis.org/A305413}{A305413}
in OEIS).

The same patterns holds for other values of $k$.

There might be other patterns regarding the number ``29", ``76", or even larger. We'll leave this for the readers to explore and discover.

\section{Collection of New Theorems}

\indent For readers' convenience, we collect all the new theorems that have been stated and proved in this article.

\textbf{Theorem \ref{a}:}

For all $k\geq 0$ and $k\in \mathbb{Z}$, $G_0 = k$, $G_1 = 1$, and $G_n = G_{n-1} + G_{n-2}$. For all $m\geq 0 \in \mathbb{Z}$, $a_m = 2k+3$ and $a_i = 4$ for all $0\leq i < m$,
$$
[4,4, ... ,4,2k+3] = \frac{G_{3m+4}}{G_{3m+1}}.
$$

\textbf{Theorem \ref{b}:}

$F_n$ is the Fibonacci sequence defined with $F_0 = 0$, $F_1 = 1$, and $F_n = \textrm{sum of two previous terms}$. For all $k \geq 0 \in \mathbb{Z}$, $a_m = 2k+3$, $a_i = 4$ for all $0 \leq i < m$,
$$
[4,4, ... , 4,2k+3] = \frac{F_{3m+4} + k\cdot F_{3m+3}}{F_{3m+1} + k\cdot F_{3m}}.
$$

\textbf{Theorem \ref{c}:}

$F_n$ is the Fibonacci sequence defined with $F_0 = 0$, $F_1 = 1$, and $F_n = \textrm{sum of two previous terms}$. For all $k \geq 1 \in \mathbb{Z}$, $a_m = k$, $a_i = 1$ for all $0 \leq i < m$,
$$
[1,1, ... ,1,k] = \dfrac{F_{m+2}+(k-1)\cdot F_{m+1}}{F_{m+1}+(k-1)\cdot F_m}.
$$

\textbf{Theorem \ref{d}:}

$F_n$ is the Fibonacci sequence defined with $F_0 = 0$, $F_1 = 1$, and $F_n = \textrm{sum of two previous terms}$. For $m \geq 0 \in \mathbb{Z}$, $a_m = 3$, $a_i = 11$ for all $0 \leq i < m$,
$$
[11,11, ... ,11,3] = 
\dfrac{F_{5m+4}}{F_{5m-1}}.
$$

\textbf{Theorem \ref{e}:}

Let $l_n$ be defined with $l_0 = 1$, $l_1 = 2$, $l_2 = 3$, $l_3 = 4$, and $l_n = l_{n-1}+l_{n-2}$ for $n\geq 2$ and let $a_i = 11$ for all $m \geq 0 \in \mathbb{Z}$ and $0 \leq i \leq m$, then
$$
[11,11, ... ,11,11] = 
\dfrac{l_{5m+5}-l_{5m-5}}{l_{5m}-l_{5m-10}}.
$$

\textbf{Theorem \ref{f}:}

Let $F_0 = 0$, $F_1 = 1$ and $F_n = F_{n-1}+F_{n-2}$. Let $a_i = 11$ for all $m \geq 0 \in \mathbb{Z}$ and $0 \leq i \leq m$, then
$$
[11,11, ... ,11,11] = \dfrac{F_{5m+10}}{F_{5m+5}}.
$$

\textbf{Theorem \ref{g}:}

Let $a_n = \dfrac{F_{3n}}{2}$ for $n \geq 0$. Let $a_i = 4$ for all $m \geq 0 \in \mathbb{Z}$ and $0 \leq i \leq m$, then
$$
[4,4, ... ,4,4] = \dfrac{a_{m+2}}{a_{m+1}}.
$$

\textbf{Theorem \ref{h}:}

Let $a_n = \dfrac{F_{7n}}{13}$ for $n \geq 0$. Let $a_i = 29$ for all $m \geq 0 \in \mathbb{Z}$ and $0 \leq i \leq m$, then
$$
[29,29, ... ,29,29] = \dfrac{a_{m+2}}{a_{m+1}}.
$$

\section{Conclusion}

\indent In this article, we mainly discussed the close connections between finite continued fractions, $F_n$ and $L_n$. For some of the theorems, we need to adjust the index numbers for the sake of expression, such as $l_n$, but this move doesn't affect the whole picture. During the process of proving, we also discovered some new formulas which connect different sequence together, such as $l_m-l_{m-10} = \dfrac{1}{5}F_{m+5}$. In order to generalize our equation, we even created some new formulas such as $G_n^k$. For those expressions that have large coefficients, we patiently use some identities from the book ``Proofs that Really Count" \cite{BQ} to assist our proof. Still, even more fascinating and impressive identities await, but we cannot include all of them in this paper. After reading this article, one might wonder what if the elements inside the continued fraction are irrational numbers, such as $\sqrt{3}$ and $\sqrt{5}$, or complex numbers involving $i$. Maybe these fraction will have some sort of connection with Binet's formula, which expresses $F_n$ and $L_n$ with irrational numbers. Curious readers might dig into this topic and explore.

\end{document}